\documentclass[12pt, a4paper]{amsart}
\usepackage{amssymb,fullpage,kotex}
\usepackage{amsmath,esint}
\usepackage[colorlinks, allcolors=RedViolet]{hyperref}
\usepackage{colortbl}
\usepackage[dvipsnames]{xcolor}

\newtheorem{theorem}[equation]{Theorem}%[section]
\newtheorem{lemma}[equation]{Lemma}
\newtheorem{proposition}[equation]{Proposition}

\theoremstyle{definition}
\newtheorem{definition}[equation]{Definition}

\theoremstyle{remark}
\newtheorem{remark}[equation]{Remark}
\numberwithin{equation}{section}

\newcommand{ \mr }{ \mathbb{R} }
\newcommand{ \R }{ \mathbb{R} }

\renewcommand{\epsilon}{\varepsilon}
\renewcommand{\phi}{\varphi}
\renewcommand{\le}{\leqslant}
\renewcommand{\ge}{\geqslant}
\renewcommand{\leq}{\leqslant}
\renewcommand{\geq}{\geqslant}
\renewcommand{\div}{\operatorname{div}}

\newcommand{\RN}[1]{%
  \textup{\uppercase\expandafter{\romannumeral#1}}%
}

\newcommand{\ainc}[1]{\hyperref[ainc]{{\normalfont(aInc){\ensuremath{_{#1}}}}}}
\newcommand{\adec}[1]{\hyperref[adec]{{\normalfont(aDec){\ensuremath{_{#1}}}}}}
\newcommand{\inc}[1]{\hyperref[ainc]{{\normalfont(Inc){\ensuremath{_{#1}}}}}}
\newcommand{\dec}[1]{\hyperref[adec]{{\normalfont(Dec){\ensuremath{_{#1}}}}}}

\newcommand{\wMAe}[1]{\hyperref[wMAe]{{\normalfont(wMA){\ensuremath{_{#1}}}}}}

\begin{document}

\title{Gradient estimates for parabolic problems with Orlicz growth and discontinuous coefficients}

% Information for first author
\author{Jehan Oh}
 % Address of record for the research reported here

\address{Department of Mathematics, Kyungpook National University, 80 Daehakro, Bukgu, Daegu 41566, Republic of Korea}
\email{jehan.oh@knu.ac.kr}

\author{Jihoon Ok}
 % Address of record for the research reported here
\address{Department of Mathematics, Sogang University, 35 Baekbeom-ro, Mapo-gu, Seoul 04107, Republic of Korea}
\email{jihoonok@sogang.ac.kr}

% \thanks will become a 1st page footnote.
\thanks{J.Oh is supported by NRF-2020R1A4A1018190. J. Ok is supported by the Sogang University Research Grant of 202010022.01.}

% General info

%\date{\today}



\begin{abstract}
We obtain Calder\'on--Zygmund type estimates for parabolic equations with Orlicz growth, where nonlinearities involved in the equations may be discontinuous for the space and time variables. In addition, we consider parabolic systems with the Uhlenbeck structure.
\end{abstract}

\maketitle

%%%%%%%%%%%%%%%%%%%%%%%%%%%%%%%%%%%%%%%%%%%%%%%%%%%%%%%%%%%%%%%%%%%%%%%%
%%%%%%%%%%%%%%%%%%%%%%%%%%%%%%%%%%%%%%%%%%%%%%%%%%%%%%%%%%%%%%%%%%%%%%%%
%%%%%%%%%%%%%%%%%%%%%%%%%%%%%%%%%%%%%%%%%%%%%%%%%%%%%%%%%%%%%%%%%%%%%%%%

\section{Introduction}\label{sec1}

We study the local regularity theory for weak solutions to the following parabolic equations  with  a general Orlicz growth condition:
\begin{equation}\label{maineq}
\partial_t u-\mathrm{div}\, A(z,Du) = -\div \left(\frac{g(|F|)}{|F|}F\right)\quad \text{in }\ \Omega_I:=\Omega\times I,
\end{equation}
where $\Omega\subset\R^n$ ($n\geq 2$) is an open set, 
$I\subset\R$ is an interval, 
$z=(x,t)\in \Omega\times I=\Omega_I$, 
$u$ is a real valued function,
and $Du\in \R^n$ is the gradient of $u$ with respect to the space variable $x$ (i.e., $Du=D_x u$).
Here, $g:[0,\infty)\to [0,\infty)$ where $g \in C([0,\infty))\cap C^1((0,\infty))$ and  $g(0)=0$ satisfies that  
\begin{equation}\label{phi-condition}
p-1 \le \frac{s g'(s)}{g(s)} \le q-1, \ \ \ ^\forall s > 0 \quad \text{for some }\ \frac{2n}{n+2} < p  \leq q,
\end{equation}
and   $A:\Omega_I\times \R^n \to \R^{n}$ where $A(z,\cdot)\in C^1(\R^n\setminus\{0\})$ for all $z\in \Omega_I$ satisfies that
\begin{equation}\label{A-condition1}
|A(z,\xi)|+ |\xi||D_\xi A(z,\xi)| \leq \Lambda g(|\xi|)
\end{equation}
 and
\begin{equation}\label{A-condition2}
D_\xi A(z,\xi) \eta \cdot \eta  \geq \nu g'(|\xi|)|\eta|^2
\end{equation}
%\begin{equation}\label{A-condition2}
%(A(z,\xi_1)-A(x,\xi_2))\cdot(\xi_1-\xi_2)  \geq \nu g'(|\xi_1|+|\xi_2|)|\xi_1-\xi_2|^2
%\end{equation}
for all $z\in\Omega_I$, $\xi,\eta\in \R^{n}\setminus\{0\}$ and for some $0<\nu\le \Lambda$, where $D_\xi A$ is the gradient of $A$ with respect to $\xi$. We further define the following:
\begin{equation}\label{phidef}
\phi(s):=\int_{0}^s g(\sigma)\,d\sigma, \quad s\ge 0.
\end{equation}  
The prototype of \eqref{maineq} is the following $g$-Laplace type equation with a coefficient:
\begin{equation}\label{prototype}
\partial_t u-\mathrm{div}\left(a(z)\frac{g(|Du|)}{|Du|}Du\right) = -\div \left(\frac{g(|F|)}{|F|}F\right)\ \  \left(\text{i.e.,}\ A(z,\xi)=a(z) \frac{g(|\xi|)}{|\xi|} \xi\right),
\end{equation}
 where $\nu \le a(\cdot)\le \Lambda$. In particular, if we set $g(s)=s^{p-1}$ and $a(\cdot)\equiv 1$,
it becomes a parabolic $p$-Laplace equation. 
The lower bound $\frac{2n}{n+2}$ of $p$ in  \eqref{phi-condition} is generally assumed even in the regularity theory for parabolic $p$-Laplace problems (see \cite{DiBenedetto_book} and \cite{AcerMin07,DiBeFrie84,KiLewis00}). 
Moreover, if $1<p\le\frac{2n}{n+2}$, even the boundedness of a weak solution requires an additional integrability assumption (see \cite{DiBenedetto_book}).

Under the above setting, we say function $u$ is a local weak solution to \eqref{maineq} if 
$u\in L^\infty_{\mathrm{loc}}(I,L^2_{\mathrm{loc}}(\Omega))\cap L^1_{\mathrm{loc}}(I,W^{1,1}_{\mathrm{loc}}(\Omega))$ 
with $%\phi(|u|),
\phi(|Du|)\in L^1_{\mathrm{loc}}(I,L^1_{\mathrm{loc}}(\Omega))$ and it satisfies 
\begin{equation}\label{weakform}
-\int_{\Omega_I} u \,\partial_t \zeta\, dz +\int_{\Omega_I} A(z,Du)\cdot D\zeta \,dz =\int_{\Omega_I} \frac{g(|F|)}{|F|}F\cdot D\zeta \,dz
\end{equation}
for all $\zeta\in C^\infty_0(\Omega_I)$.  
Then, the main result of this paper is to prove  that the following implication holds for every  weak $\Phi$-function $\psi=\psi(s)$ satisfying \ainc{p_1} and \adec{q_1} for some $1<p_1\le q_1$ (see Section~\ref{subsec2.2} for the definitions of the weak $\Phi$-function, \ainc{} and \adec{}):
\begin{equation}\label{implication}
\phi(|F|) \in L^\psi_{\mathrm{loc}}(\Omega_I) \ \ \Longrightarrow\ \ \phi(|Du|) \in L^\psi_{\mathrm{loc}}(\Omega_I).
\end{equation}
The results is demonstrated by obtaining a parabolic Calder\'on--Zygmund type estimate under a suitable discontinuous assumption of the nonlinearity $A$ for the  $z$ variable.

In the classical case in which $\phi(s)=s^p$, $g(s)=ps^{p-1}$, and $\psi(s)=s^q$, where $p>\frac{2n}{n+2}$ and  $q>1$, the implication of \eqref{implication} was proved by \cite{AcerMin07}. The the proof took advantage of the higher integrability in \cite{KiLewis00} and the Lipschitz regularity of parabolic $p$-Laplace problems in \cite{DiBeFrie84,DiBenedetto_book}, and the so-called large-$M$-inequality principle. We also refer to \cite{DiBeFrie85,DiBeGiaVe08,DiBeGiaVe_book} regarding the H\"older continuity and Harnack's inequality, and \cite{Baroni13,Bogelein14,BOR13} for Calder\'on--Zygmund type estimates. 

After determining these results for the classical case, research  has been and still is being conducted to obtain the regularity theory for parabolic problems with Orlicz growth. 
Lieberman first studied the systematic regularity theory for elliptic equations with Orlicz growth in \cite{Lieberman91} and later considered parabolic problems with Orlicz growth in \cite{Lieberman06}.  
For further regularity results for parabolic problems with Orlicz growth, we refer to \cite{BaLin17,Cho18-2,DieScharSchwa19,HasOkparahigh,HwangLie15-1,HwangLie15-2,Yao19} and related references. The implication of \eqref{implication} was proved in \cite{Cho18-2} for \eqref{maineq} with the special case $A(z,\xi)=\frac{g(|\xi|)}{|\xi|}\xi$ (i.e., \eqref{prototype} with $a(\cdot)\equiv 1$), and $p$ in \eqref{phi-condition} is greater than or equal to $2$. Hence, in this paper, we consider general nonlinearities that depend on $z$ and moreover may be discontinuous for $z$.  We also refer to \cite{Verde11,BC16,Cho18-1} for Calder\'on--Zygmund type estimates in the Orlicz setting for elliptic problems.

Ws introduce our main result. A parabolic cylinder $Q_{r,\rho}(z_0)$ where $z_0=(x_0,t_0)\in \R^n\times\R$ is denoted as $Q_{r,\rho}(z_0):=B_r(x_0)\times (t_0-\rho^2, t_0+\rho^2)$, where $B_r(x_0)$ is the open ball in $\R^n$ with the center $x_0$ and  radius $r>0$. For simplicity, we write $Q_{r}(z_0)=Q_{r,r}(z_0)$ and $Q_{r,\rho}=Q_{r,\rho}(z_0)$ and $Q_{r}=Q_{r}(z_0)$, if the center is obvious.  The main regularity assumption on $A$ is the following.
\begin{definition}\label{def:vanishing}
Let $\delta,R>0$. $A:\Omega_I\times \R^n\to \R^n$ is  \textit{$(\delta,R)$-vanishing} if the following holds for all $Q_{r,\rho}\subset \Omega_I$ with $r,\rho\in (0,R]$:
\[
\fint_{Q_{r,\rho}} \theta(A; Q_{r,\rho})(z)\,dz \le \delta,
\]
where
\[
\theta(A; Q_{r,\rho})(z):=\sup_{\xi\in\R^n \setminus \{ 0\}} \frac{|A(z,\xi)- (A(\cdot,\xi))_{Q_{r,\rho}}|}{g(|\xi|)},
\ \ \text{and}\ \ (A(\cdot,\xi))_{Q_{r,\rho}}=\fint_{Q_{r,\rho}} A(z,\xi)\, dz.
\]
\end{definition}
We note that from \eqref{A-condition1}, $\theta \le 2\Lambda$; hence the $(\delta,R)$-vanishing condition implies that for any $\kappa>1$,
\begin{equation}\label{property_vanishing}
\fint_{Q_{r,\rho}} \left[ \theta(A; Q_{r,\rho})(z)\right]^{\kappa} dz \le (2\Lambda)^{\kappa-1} \delta
\quad \text{for all } \ Q_{r,\rho}\subset \Omega_I \ \text{ with }\ r,\rho\in (0,R].
\end{equation}

 In the special case in which $A(z,\xi)=a(x)b(t)A_0(\xi)$, where $z=(x,t)$, the $(\delta,R)$-vanishing condition can be implied by the local BMO(bounded mean oscillation) conditions of $a(\cdot)$ and $b(\cdot)$. Next, we state the main theorem in this paper.

\begin{theorem}\label{mainthm}
Let $g:[0,\infty)\to [0,\infty)$ with $g \in C([0,\infty))\cap C^1((0,\infty))$ and  $g(0)=0$ satisfy  \eqref{phi-condition}, and $A:\Omega_I\times \R^n\to\R^n$ with $A(z,\cdot)\in C^1(\R^n\setminus\{0\})$ for all $z\in\Omega_I$ do \eqref{A-condition1} and \eqref{A-condition2}. 
In addition, let $F\in L^\phi_{\mathrm{loc}}(\Omega_I;\R^n)$ with $\phi$ defined in \eqref{phidef}, and $\psi:[0,\infty)\to [0,\infty)$ be a weak $\Phi$-function satisfying \ainc{p_1} and \adec{q_1} for some $1<p_1\le q_1$ with constant $L\ge 1$. 
There exists small $\delta=\delta(n,\nu,\Lambda,p,q,p_1,q_1,L,g(1),\psi(1))>0$ such that if $A$ is $(\delta,R_0)$-vanishing for some $R_0>0$ and $u$ is a local weak solution to \eqref{maineq}, then we have the following implication:
$$
\phi(|F|) \in L^\psi_{\mathrm{loc}}(\Omega_I) \ \ \Longrightarrow\ \ \phi(|Du|) \in L^\psi_{\mathrm{loc}}(\Omega_I)
$$
with the following estimate: for any $Q_{2R}=Q_{2R}(z_0)\Subset \Omega_I $ with $R\le R_0$,
\[%begin{equation}\label{highesti}
\begin{split}
\fint_{Q_R} \psi(\phi(|Du|)) \,dz & \leq c \left[ \Psi \bigg( \fint_{Q_{2R}} \left[ \phi(|Du|) + \phi(|F|) \right] dz \bigg) \right] \fint_{Q_{2R}} \phi(|Du|) \, dz \\
& \qquad + c \fint_{Q_{2R}} \psi(\phi(|F|)) \, dz
\end{split}
\]%end{equation}
for some $c=c(n,\nu,\Lambda,p,q,p_1,q_1,L,g(1),\psi(1))>0$, where $Q_R=Q_R(z_0)$, $\Psi(s):=( \psi_1\circ\phi\circ \mathcal{D}^{-1})(s)$,  $\psi_1(s):= \frac{\psi(s)}{s}$, and $\mathcal{D}^{-1}$ is the inverse of 
\begin{equation}\label{D-def}
\mathcal{D}(s):=\min\{s^2,\phi(s)^{\frac{n+2}{2}}s^{-n}\}=\min\{1,\phi_2(s)^{\frac{n+2}{2}}\}s^2, \quad \text{where }\ \phi_2(s):=\frac{\phi(s)}{s^2}.
\end{equation}  
\end{theorem}

\begin{remark} Regarding the function $\mathcal D$, since $\phi$ satisfies \inc{p} with $p>\frac{2n}{n+2}$, $\mathcal D$ is strictly increasing; hence, the inverse $\mathcal D^{-1}$ is well-defined. In the case of the power functions, $g(s)=ps^{p-1}$ (hence, $\phi(s)=s^p$) and $\psi(s)=s^q$,  we have  $\mathcal{D}(s)=\min\{s^2,s^{\frac{p(n+2)}{2}-n}\}$
%$\mathcal{D}(t)=t^2 +t^{\frac{p(n+2)}{2}-n}$
and, thus,
\[
\Psi(s)
= \left(\max\{s^{\frac{p}{2}},s^{\frac{2p}{p(n+2)-2n}}\}\right)^{q-1}.
%\approx \min\{t^{\frac{p}{2}},t^{\frac{2p}{p(n+2)-2n}}\}.
\]
Therefore, our result exactly implies the known results for the $p$-growth case (see, e.g., \cite{BoDuMin11}). 
\end{remark}
\begin{remark}
We present examples of a function $g$ (and, hence, $\phi$ defined in \eqref{phidef}) which are not just power functions.  Simple examples are $g(s)= s^{p-1}[\log(1+s)]^q$ and $g(s)=\int_0^s\min\{u^{p-2},u^{q-2}\}\,du$. The following example is quite complicated.
Let $p=2-\frac{3}{2n}$ and $q=2+\frac{3}{2n}$, and set $\kappa=(q-p)/3>0$. We define the sequence $s_k=2^{2^k}$ for $k=0,1,2,\cdots$, and the function
$$
g(s) =  \left\{ \begin{array}{ll}
s^{p-1+\kappa}, & 0<s<1, \\
2^{(p-1+\kappa)(s-1)}, & 1 \leq s <2, \\
s_{2k+1}^{-\kappa} s^{q-1}, & s_{2k} \leq s < s_{2k+1}, \\
s_{2k+2}^{\kappa} s^{p-1}, & s_{2k+1} \leq s < s_{2k+2}.
\end{array} \right.
$$
It oscillates between degenerate and singular behavior (see \cite{BaLin17} for more details).
\end{remark}

Finally, we shall introduce techniques used in the proof of the theorem. Our approach is based on the so-called \textit{large-M-inequality principl}e that was used in the case in which $\phi(s)=s^p$ and $\psi(s)=s^q$ in \cite{AcerMin07} (we also refer to \cite{Min07} for its origin). In addition, we modify this principle in the Orlicz setting. A comparison estimate  is required between the main equation \eqref{maineq} and a homogeneous equation with a nonlinearity being independent of $z$. To our best knowledge, a higher integrability result  is essentially needed priori to obtain such a comparison estimate for partial differential equations with  BMO-type discontinuous coefficients. Recently, H\"ast\"o and the second author of this paper proved the higher integrability for parabolic problems with Orlicz growth, hence this leads us to prove the main theorem. Moreover, we also use the Lipschitz regularity of the homogenous problem that has been proved in \cite{BaLin17,DieScharSchwa19}.

The remaining paper is organized as follows. Section \ref{sec2} introduces the notation, 
Orlicz functions with related inequalities and function spaces, and preliminary results. In Section \ref{sec3}, we derive the comparison estimates.  In Section \ref{sec4}, we prove Theorem~\ref{mainthm}. Finally, we briefly discuss  the implication for parabolic systems with Orlicz growth in the final section, Section \ref{sec5}.

%%%%%%%%%%%%%%%%%%%%%%%%%%%%%%%%%%%%%%%%%%%%%%%%%%%%%%%%%%%%%%%%%%%%
%%%%%%%%%%%%%%%%%%%%%%%%%%%%%%%%%%%%%%%%%%%%%%%%%%%%%%%%%%%%%%%%%%%%
%%%%%%%%%%%%%%%%%%%%%%%%%%%%%%%%%%%%%%%%%%%%%%%%%%%%%%%%%%%%%%%%%%%%

\section{Preliminaries}\label{sec2}

%%%%%%%%%%%%%%%%%%%%%%%%%%%%%%%%%%%%%%%%%%%%%%%%%%%%%%%%%%%%%%%%%%%%
\subsection{Notation}\label{subsec2.1}

Let $w=(y,\tau)\in \mr^n\times \mr$ and $r>0$.  Then, $Q_r(w)=B_r(y)\times (\tau-r^2,\tau+r^2)$ is a usual parabolic cylinder with base $B_r(y):= \{x\in \R^n: |x-y|<r\}$. We further define a so-called intrinsic parabolic cylinder with $\lambda>0$ and  function $\phi:(0,\infty)\to(0,\infty)$ by
\[
Q^{\lambda}_{r}(w):=B_r(y)\times (\tau-\tfrac{r^2}{\phi_2(\lambda)},
\tau+\tfrac{r^2}{\phi_2(\lambda)}),\quad
\text{where }\ \phi_2(s):=\frac{\phi(s)}{s^2}.
\] 

A function $f:I \to \R$ with $I\subset \R$ is  \textit{almost increasing}  if, for some $L\ge 1$, $f(t)\leq Lf(s)$  for every $t,s\in I$ with $t<s$. If we set $L=1$, then we say that $f$ is increasing.  
Similarly, we define an \textit{almost decreasing} or decreasing function. 

For an integrable function $f: U\to \R^m$, $U\subset \R^N$, we define $(f)_{U} := \fint_U f\, dz := \frac{1}{|U|}\int_U f \,dz$, where $|U|$ is the Lebesgue measure of $U$ in $\R^N$. 

 We write $f\lesssim g$ if $f\le c g$ for some $c>0$, and $f\approx g$ if $c^{-1}g\le f\le c g$ for some $c\ge 1$.

%%%%%%%%%%%%%%%%%%%%%%%%%%%%%%%%%%%%%%%%%%%%%%%%%%%%%%%%%%%%%%%%%%%%

\subsection{Orlicz functions}\label{subsec2.2}
Let $\phi:(0,\infty)\to[0,\infty)$ and $p,q>0$. 
We introduce the following conditions:
\begin{itemize}
\item[\normalfont(aInc)$_p$]\label{ainc} 
The map $(0,\infty)\ni s\mapsto \phi(s)/s^p$ is almost increasing with constant $L\geq 1$.
\vspace{0.2cm}
\item[\normalfont(aDec)$_q$]\label{adec} 
The map $(0,\infty)\ni s\mapsto \phi(s)/s^q$ is almost decreasing with constant $L\geq 1$.
\end{itemize}

In particular, when $L=1$ we use \inc{} and \dec{}, instead of \ainc{} and \adec{}, respectively.

From the above definition, we directly deduce that if  $\phi$ satisfies \ainc{p} and \adec{q} for some $p\le q$ then 
\[
 c^q L^{-1} \phi(s) \leq  \phi(cs)\leq  c^pL \phi(s)
 \quad \text{and}\quad 
 C^pL^{-1} \phi(s) \leq  \phi(Cs)\leq  C^qL \phi(s)
\]
for every  $s\in(0,\infty)$ and every $0<c <1 <C$. Moreover, if $p_1<p_2$, or $q_1<q_2$, then  \ainc{p_2} implies \ainc{p_1}, or \adec{q_1} does \adec{q_2}.  We use these properties many times later without explicit mention.

Next, we introduce the definition of a weak $\Phi$-function referring to \cite[Definition 2.1.3]{HarH_book}.
\begin{definition} 
The function $\phi:[0,\infty)\to[0,\infty]$ is a \emph{weak $\Phi$-function} 
if it is increasing with $\phi(0)=0$, $\displaystyle \lim_{s\to 0^+}\phi(s)=0$, $\displaystyle \lim_{s\to\infty}\phi(s)=\infty$ 
and it satisfies \ainc{1}.  Moreover, if a weak $\Phi$-function is left-continuous and convex, then it is a \emph{convex $\Phi$-function}. 
\end{definition}
If $\phi$ is a weak $\Phi$-function or a convex $\Phi$-function, we write $\phi\in \Phi_{\mathrm{w}}$ or $\phi\in \Phi_{\mathrm{c}}$, respectively. We notice that $\sqrt{\phi(s^2)}$ need not be 
convex even if $\phi$ is convex, but the \ainc{1} property is conserved. For a weak $\Phi$-function, we define its conjugate function as follows:
\[
\phi^*(r) := \sup_{s\in[0,\infty)} (rs-\phi(s)), \quad r\in[0,\infty).
\]
The definition directly implies Young's inequality:
\begin{equation}\label{Young_1}
rs \le \phi(s) + \phi^*(r), \quad r,s\in[0,\infty).
\end{equation}
We next introduce the properties of regular $\Phi$-functions.

\begin{proposition}\label{prop0} 
Let $\gamma>0$ and $0<p\le q$. Suppose that $\phi\in\Phi_{\mathrm{c}}\cap C^1([0,\infty))$.
\begin{itemize}
\item[(1)] 
If $\phi'$ satisfies \ainc{\gamma}, \adec{\gamma}, \inc{\gamma}, or \dec{\gamma}, then $\phi$ satisfies \ainc{\gamma+1}, \adec{\gamma+1}, \inc{\gamma+1}, or \dec{\gamma+1}, respectively, with the same constant $L\geq 1$. 

\item[(2)] 
If $\phi'$ satisfies \adec{\gamma} with constant $L\ge 1$, 
then
\[
\phi(s) \leq s\phi'(s)  \leq 2^{\gamma+1}L \phi(s)\quad \text{for all }\ s\in [0,\infty).
\]
\item[(3)] $\phi(s)$ satisfies \inc{p} and \dec{q} if and only if 
\begin{equation}\label{equiv2}
p \phi(s)\le s\phi'(s) \le q \phi(s)  \quad \text{for all }\ s\in [0,\infty).
\end{equation}
\item[(4)] 
$\phi^*(\phi'(s))\le s\phi'(s)$. 

\end{itemize}
\end{proposition}
\begin{proof}
We only prove the statement in (3). The rest of the properties can be found in \cite[Proposition 3.6]{HasOk}. 

We first suppose that $\phi$ satisfies \inc{p} and \dec{q}. Then, for $s,h>0$ we have the following:
\[
\frac{\left(\frac{h+s}{s}\right)^p-1}{h}\phi(s)
\le \frac{\phi(s+h)-\phi(s)}{h} \le 
\frac{\left(\frac{h+s}{s}\right)^q-1}{h}\phi(s).
\]
Passing $h\to 0$, we have the inequalities in \eqref{equiv2}. For the converse, from \eqref{equiv2},  $\frac{d}{ds} (\frac{\phi(s)}{s^p})>0$ and $\frac{d}{ds} (\frac{\phi(s)}{s^q})<0$, which imply \inc{p} and \dec{q}, respectively.
\end{proof}

From the above proposition, if $g$ satisfies \eqref{phi-condition}, then $\phi$ defined in \eqref{phidef} satisfies \inc{p} and \dec{q}.
From \eqref{Young_1} and Proposition \ref{prop0}, for any $\epsilon \in (0,1)$,
\begin{equation}\label{Young_2}
rs \le 
\begin{cases} 
\epsilon \phi(s) + c(\epsilon) \phi^*(r)\\
\epsilon \phi^*(s) + c(\epsilon) \phi(r)
\end{cases}
 , \quad r,s\in[0,\infty),
\end{equation}
for some constant $c(\epsilon)=c(p,q,\epsilon)>1$.
%Similarly, we see that for any $\epsilon \in (0,1)$,
%\begin{equation}\label{Young_3}
%rg(s) \le \epsilon \phi(s) + c(\epsilon) \phi(r), \quad r,s\in[0,\infty),
%\end{equation}
%for some constant $c(\epsilon)=c(p,q,\epsilon)>1$.

For $\phi\in \Phi_{\mathrm{w}}$ satisfying \ainc{p} and \adec{q} for some $1<p\le q$ and $\Omega\subset \R^n$, we define the \textit{Orlicz space} as follows:  
\[
L^\phi(\Omega):= \left\{f\in L^0(\Omega): \int_\Omega \phi(|f(x)|)\,dx < \infty \right\},
\]
where $L^0(\Omega)$ is the set of measurable functions in $\Omega$. For more properties for $\Phi$-functions and Orlicz spaces, we refer to \cite{HarH_book}.

Next we introduce a vector field $V: \mr^n \to \mr^n$, which is defined by
%\begin{equation}\label{def_of_A}
%A(\xi) := \frac{g(|\xi|)}{|\xi|} \, \xi
%\end{equation}
%and
\[
V(\xi) := \sqrt{\frac{g(|\xi|)}{|\xi|}} \, \xi,
\]
where $g\in C([0,\infty)) \cap C^1((0,\infty))$ satisfies \eqref{phi-condition}. For simplicity we assume that $g(1)=1$. Then we otain the following relation between the above vector field and the function $\phi$ defined by \eqref{phidef}.

\begin{lemma}[\cite{DieEtt08}]\label{lem:relation_V}
In the above setting, we have
\begin{equation}\label{relation_V_1}
g'(|\xi_1|+|\xi_2|) |\xi_1 - \xi_2|^2 \approx \left| V(\xi_1) - V(\xi_2) \right|^2
\end{equation}
and
\begin{equation}\label{relation_V_2}
\phi(|\xi_1-\xi_2|) \lesssim \kappa \left| V(\xi_1) - V(\xi_2) \right|^2 + c(\kappa)\phi(|\xi_2|)
\end{equation}
for all $\xi_1, \xi_2 \in \mr^n$ and $\kappa>0$, where the hidden constants depend only on $p$, $q$ and $n$.
In addition,
\begin{equation}\label{relation_V_3}
\left| V(\xi) \right|^2 \approx \phi(|\xi|)
\end{equation}
holds for all $\xi \in \mr^n$.
\end{lemma}

Let the vector field $A(z,\xi)$ satisfy \eqref{A-condition1} and \eqref{A-condition2} with \eqref{phi-condition}.
Then, from \eqref{A-condition2}, we have the following monotonicity property of the vector field $A(z,\xi)$ with respect to the gradient variable $\xi$:
\begin{equation}\label{relation_A_V}
\left| V(\xi_1) - V(\xi_2) \right|^2 \overset{\eqref{relation_V_1}}{\approx} g'(|\xi_1|+|\xi_2|) |\xi_1 - \xi_2|^2  \lesssim \left( A(z,\xi_1) - A(z,\xi_2) \right) \cdot (\xi_1 - \xi_2) 
\end{equation}
for all $z\in\Omega_I$ and $\xi_1, \xi_2 \in \mr^n$, where the hidden constants depend on $p$, $q$, $n$, and $\nu$.

%%%%%%%%%%%%%%%%%%%%%%%%%%%%%%%%%%%%%%%%%%%%%%%%%%%%%%%%%%%%%%%%%%%%
\subsection{Preliminary regularity results}

In this subsection we introduce regularity results for homogeneous parabolic problems with Orlicz growth. The first result is the higher integrability of a weak solution.

\begin{theorem}[Higher integrability, {\cite[Theorem 1.5]{HasOkparahigh}}]\label{thmhigh}
Let $\phi:[0,\infty)\to[0,\infty)$ be a weak $\Phi$-function satisfying \ainc{p} and \adec{q} 
with constant $L\geq 1$ and let $w$ be a  weak solution to 
\[
w_t - \mathrm{div}\, A(z,Dw) =0 \quad \text{in }\ Q_{2r}, 
\]
where $A: Q_{2r}\times \R^n \to \R^n$ satisfies that
\[%begin{equation}\label{A-condition3}
|A(z,\xi)|\le \Lambda \frac{\phi(|\xi|)}{|\xi|} \quad \text{and}\quad A(z,\xi) \cdot \xi \ge \nu \phi(|\xi|)
\]%end{equation}
for every $z\in Q_{2r}$, $\xi\in \R^n \setminus \{0\}$ and for some $0<\nu\le \Lambda$. There exists $\sigma=\sigma(n,\nu,\Lambda,p,q,L,\phi(1))>0$ such that $\phi(|Dw|)\in L^{1+\sigma}(Q_r)$ with the following estimate:
\[%begin{equation}\label{highesti}
\fint_{Q_{r}} \phi(|Dw|)^{1+\sigma}\,dz \leq c \left[(\phi \circ \mathcal{D}^{-1})\bigg(\fint_{Q_{2r}}\phi(|Dw|)\,dz\bigg)\right]^{\sigma} \fint_{Q_{2r}}\phi(|Dw|)\,dz
\]%end{equation}
for some $c=c(n,\nu,\Lambda,p,q,L,\phi(1))>0$, where $\mathcal{D}$ is given in \eqref{D-def} and $\mathcal{D}^{-1}$ is the inverse of $\mathcal D$.
\end{theorem}

The next result is the Lipschitz regularity for homogeneous equations with nonlinearity independent of $z$. For Cauchy--Dirichlet problems, we refer to {\cite[Theorem 1.2]{BaLin17}}. Using this result with mollification, we have the following interior Lipschitz regularity  for general problems.

\begin{theorem}[Lipschitz regularity]\label{thmLip}
Let $g:[0,\infty)\to [0,\infty)$ with $g \in C([0,\infty))\cap C^1((0,\infty))$ and  $g(0)=0$ satisfy \eqref{phi-condition}, and $A_0: \R^n\to\R^n$ with $A_0\in C^1(\R^n\setminus\{0\})$ do \eqref{A-condition1} and \eqref{A-condition2}, replacing $A(z,\xi)$ with $A_0(\xi)$.  If  $v$ is a weak solution to
\[
v_t -\mathrm{div}\, A_0(Dv) =0 \quad \text{in }\ Q_{3r} ,\\
\]
then $v$ belongs to $L^\infty(Q_r)$ with the following estimate:
\[
\|Dv\|_{L^\infty(Q_r)} \leq c \left(\fint_{Q_{2r}} [\phi(|Dv|)+1] \, dz\right)^{\max\{\frac{1}{2},\frac{2}{p(n+2)-2n}\}}
\]
for some $c=c(n,\nu,\Lambda,p,q, g(1))>0$.
\end{theorem}

\begin{proof}
Let  $0<\epsilon < \min\{r^2,r\}$, $\eta_\epsilon$ be a standard mollifier, and $v_\epsilon= v * \eta_\epsilon$. Then we know that, up to a subsequence,
\[
\left\{
\begin{array}{cccl}
Dv_\epsilon  & \longrightarrow & Dv & \text{a.e. in }\ Q_{2r}\\
\phi (|Dv_\epsilon|)  & \longrightarrow & \phi(|Dv|) & \text{in }\ L^1(Q_{2r})\\ 
\end{array}
\right.
\quad \text{as}\quad
\epsilon \ \to\ 0.
\]
Hence, we have the following:
\[
\int_{Q_{2r}}\phi (|Dv_\epsilon|)\,dx  \le 2 \int_{Q_{2r}}\phi (|Dv|)\,dx \quad \text{for ant sufficiently small }\epsilon>0. 
\]
We next consider the weak solution $h_\epsilon$ to
\[
\left\{\begin{array}{l}
(h_\epsilon)_t -\mathrm{div}\, A_0(Dh_\epsilon) =0 \quad \text{in }\ Q_{2r} ,\\
h_\epsilon=v_\epsilon \quad \text{on }\ \partial_\mathrm{p} Q_{2r}.
\end{array}\right.
\]
Then, by \cite[Theorem 1.2]{BaLin17}, we have the following: 
\[
\|Dh_\epsilon\|_{L^\infty(Q_r)} \leq c \left(\fint_{Q_{2r}} [\phi(|Dh_\epsilon|)+1] \, dz\right)^{\max\{\frac{1}{2},\frac{2}{p(n+2)-2n}\}},
\]
where the constant $c>0$ depends only on $n,\nu,\Lambda,p,q, g(1)$ and is independent of $\epsilon$ and the boundary datum $v_\epsilon$.
Next, we prove that $Dh_\epsilon\to Dv$ in $L^\phi(Q_{2r},R^n)$ as $\epsilon\to 0$, up to a subsequence. Then the proof is completed.

First, $v_\epsilon$ becomes a weak solution to the following parabolic problem:
\[
(v_\epsilon)_t -\mathrm{div}\, ([A_0(Dv)]_\epsilon) =0 \quad \text{in }\ Q_{2r},
\]
where $[A_0(Dv)]_\epsilon = A_0(Dv) * \eta_\epsilon$. Then, from the previous two equations,
\[
 \int_{Q_{2r}} (v_\epsilon-h_\epsilon)_t (v_\epsilon-h_\epsilon)  \,dz + \int_{Q_{2r}} ([A_0(Dv)]_\epsilon-A_0(Dh_\epsilon))\cdot (Dv_\epsilon-Dh_\epsilon)\,dz=0
\]
(cf. Remark~\ref{rmk:textfunction}). The first integral is 
\[
\int_{Q_{2r}} (v_\epsilon-h_\epsilon)_t (v_\epsilon-h_\epsilon)  \,dz=\int_{Q_{2r}} \frac{1}{2}\frac{d}{dt}(v_\epsilon-h_\epsilon)^2  \,dz =\frac{1}{2}\int_{B_{2r}} (v_\epsilon-h_\epsilon)^2|_{t=t_0+(2r)^2}  \,dx \ge 0.
\]
Hence, from \eqref{relation_A_V} and \eqref{Young_2}, 
\begin{equation}\label{pf1}\begin{aligned}
\int_{Q_{2r}} & |V(Dv_\epsilon)-V(Dh_\epsilon)|^2\,dz 
\le c \int_{Q_{2r}} (A_0(Dv_\epsilon)-A_0(Dh_\epsilon))\cdot (Dv_\epsilon-Dh_\epsilon)\,dz \\
&\le c \int_{Q_{2r}} (A_0(Dv_\epsilon) -[A_0(Dv)]_\epsilon)\cdot (Dv_\epsilon-Dh_\epsilon)\,dz\\
&\le c(\kappa_1) \int_{Q_{2r}} \phi^*(|A_0(Dv_\epsilon) -[A_0(Dv)]_\epsilon|)\,dz + \kappa_1\int_{Q_{2r}}\phi (|Dv_\epsilon-Dh_\epsilon|)\,dz\\
&\le c(\kappa_1) \int_{Q_{2r}} \phi^*(|A_0(Dv_\epsilon) - A_0(Dv)|)\,dz + c(\kappa_1) \int_{Q_{2r}} \phi^*(|[A_0(Dv)]_\epsilon-A_0(Dv)|)\,dz\\
&\qquad +  \kappa_1\int_{Q_{2r}}[\phi (|Dv_\epsilon|)+\phi(|Dh_\epsilon|)]\,dz,
\end{aligned}\end{equation}
where $\kappa_1\in(0,1)$. From the preceding estimate, using  \eqref{relation_V_3}, \eqref{A-condition2} and Proposition~\ref{prop0} (4), we first observe that
\[
\int_{Q_{2r}} \phi(|Dh_\epsilon|)\,dz \le c \int_{Q_{2r}} \phi(|Dv_\epsilon|)\,dz \le c \int_{Q_{2r}} \phi(|Dv|)\,dz 
\]
for any sufficiently small $\epsilon>0$. We next prove that the first integral on the right hand side in \eqref{pf1} approaches $0$ as $\epsilon\to0$, up to a subsequence.  Because $\phi^*(|A_0(Dv_\epsilon)|) \le c \phi(Dv_\epsilon)$ and  $\phi(Dv_\epsilon)\to \phi(Dv)$ in $L^1(Q_{2r})$, by a modified Lebesgue dominant convergence theorem,  $\phi^*(|A_0(Dv_\epsilon)|) \to \phi^*(|A_0(Dv)|)$ in $L^1(Q_{2r})$. Moreover,   
$A_0(Dv_\epsilon)\to A_0(Dv)$ a.e. in $Q_{2r}$; thus,  $A_0(Dv_\epsilon) \to A_0(Dv)$ in $L^{\phi^*}(Q_{2r},\R^n)$, up to a subsequence. Therefore, because $[A_0(Dv)]_\epsilon \to A_0(Dv)$ in $L^{\phi^*}(Q_{2r},\R^n)$ and $\kappa_1\in(0,1)$ is arbitrary, 
\[
\lim_{\epsilon\to0} \int_{Q_{2r}}  |V(Dv_\epsilon)-V(Dh_\epsilon)|^2\,dz =0.
\]
Finally, applying \eqref{relation_V_2}, we prove the following:
$$
\lim_{\epsilon\to0} \int_{Q_{2r}}  \phi(|Dh_\epsilon-Dv|)\,dz \le \lim_{\epsilon\to0} \int_{Q_{2r}}  \phi(|Dh_\epsilon-Dv_\epsilon|)\,dz+\lim_{\epsilon\to0} \int_{Q_{2r}}  \phi(|Dv_\epsilon-Dv|)\,dz =0.
$$
\end{proof}

\section{Comparison}
\label{sec3}
  
In this section, we derive comparison estimates. Assuming $Q_{4} = Q_{4}(0) \Subset \Omega_I$, we consider the following homogeneous problem:
\begin{equation}\label{homo_prob}
\left\{ \begin{array}{rclcl}
\partial_t w-\mathrm{div}\, A(z,Dw) & = & 0 & \ \text{in} & Q_{4}, \\
w & = & u & \ \text{on} & \partial_p Q_{4},
\end{array} \right.
\end{equation}
where $\partial_p Q_{4}$ is the parabolic boundary of $Q_4$ and $u$ solves Equation \eqref{maineq}.
Then, we obtain the following comparison estimate.

\begin{lemma}[First comparison estimate]\label{lem:comp1}
Let $w$ be the weak solution to \eqref{homo_prob}.
Then for any $\epsilon>0$, there exists a small $\delta=\delta(n,\nu,\Lambda,p,q,\epsilon)>0$  such that if
\begin{equation}\label{comp1_assumption}
\fint_{Q_{4}} \left[ \phi(|Du|) + \frac{1}{\delta} \phi(|F|) \right] dz \leq 1
\end{equation}
then
\[%begin{equation}\label{comp1_result}
\fint_{Q_{4}} \left| V(Du)-V(Dw) \right|^2 dz \leq \epsilon.
\]%end{equation}
\end{lemma}

\begin{remark}\label{rmk:textfunction}
We will use test functions involving the weak solution $u$ in the weak formulation \eqref{weakform}. However, the weak solution $u$ to the parabolic equation in \eqref{maineq} may not be differentiable in the time variable $t$.
It is standard to consider Steklov averages to overcome this issue (see \cite{DiBenedetto_book} and \cite{BaLin17}). This argument is now well-known in the field of parabolic PDEs; therefore, we abuse the notation $\partial_t u$ without further explanation.
\end{remark}

\begin{proof}[Proof of Lemma \ref{lem:comp1}]
We take $\zeta=u-w$ as a test function in \eqref{maineq} and \eqref{homo_prob} using Steklov averages to obtain
\begin{equation}\label{weakform_comp1}
\int_{Q_{4}} \left( A(z,Du) - A(z,Dw) \right) \cdot (Du-Dw) \, dz \leq \int_{Q_{4}} \frac{g(|F|)}{|F|}F \cdot (Du-Dw) \, dz.
\end{equation}
Using \eqref{relation_A_V} and Young's inequality \eqref{Young_2} with Proposition~\ref{prop0}(4), we obtain the following:
\begin{equation}\label{comp1_pf_1}
\begin{split}
& \fint_{Q_{4}} \left| V(Du)-V(Dw) \right|^2 dz \\
& \quad \leq \fint_{Q_{4}} g(|F|) (|Du|+|Dw|) \, dz \\
& \quad  \leq \tau \fint_{Q_{4}} \phi(|Du|) \, dz + \tau \fint_{Q_{4}} \phi(|Dw|) \, dz + c(\tau) \fint_{Q_{4}} \phi(|F|) \, dz 
\end{split}
\end{equation}
for any $\tau \in (0,1)$.

From~\eqref{relation_V_3}, \eqref{relation_A_V}, \eqref{A-condition1}, \eqref{weakform_comp1}, Lemma~\ref{lem:relation_V} and Young's inequalities \eqref{Young_2} with Proposition~\ref{prop0} (4), 
\[\begin{split}
\fint_{Q_{4}} \phi(|Dw|) \, dz & \le c \fint_{Q_{4}} |V(Dw)|^2 \, dz \\
& \le c \fint_{Q_{4}} A(z,Dw) \cdot Dw \, dz \\
& \le c \fint_{Q_{4}} g(|Dw|)|Du| \, dz + \fint_{Q_{4}} g(|Du|)(|Du|+|Dw|) \, dz \\
& \qquad + \fint_{Q_{4}} g(|F|) (|Du|+|Dw|) \, dz \\
& \leq \frac12 \fint_{Q_{4}} \phi(|Dw|) \, dz + c \fint_{Q_{4}} \phi(|Du|) \, dz + c \fint_{Q_{4}} \phi(|F|) \, dz.
\end{split}
\]
Hence,
\begin{equation}\label{comp1_pf_2}
\fint_{Q_{4}} \phi(|Dw|) \, dz \leq  c \fint_{Q_{4}} \phi(|Du|) \, dz + c \fint_{Q_{4}} \phi(|F|) \, dz.
\end{equation}

Combining \eqref{comp1_pf_1} and \eqref{comp1_pf_2} yields
\[\begin{split}
\fint_{Q_{4}} \left| V(Du)-V(Dw) \right|^2 dz & \leq (c+1)\tau \fint_{Q_{4}} \phi(|Du|) \, dz + c(\tau) \fint_{Q_{4}} \phi(|F|) \, dz \\
& \leq (c+1)\tau + c(\tau)\delta,
\end{split}
\]
where we used the assumption in \eqref{comp1_assumption}.
By setting $\tau \in (0,1)$ and $\delta \in (0,1)$ properly, we conclude that
\[
\fint_{Q_{4}} \left| V(Du)-V(Dw) \right|^2 dz \leq \epsilon.
\]
\end{proof}

We next consider the vector field $A_0 : \R^n\to \R^n$, which is defined by
\[%begin{equation}\label{def_A_0}
A_0(\xi) := \fint_{Q_{3}} A(z,\xi) \, dz = \frac{1}{|Q_{3}|} \int_{Q_{3}} A(z,\xi) \, dz,
\]%end{equation}
and the following reference problem:
\begin{equation}\label{freezing_prob}
\left\{ \begin{array}{rclcl}
\partial_t v-\mathrm{div}\, A_0(Dv) & = & 0 & \ \text{in} & Q_{3}, \\
v & = & w & \ \text{on} & \partial_p Q_{3},
\end{array} \right.
\end{equation}
where $w$ solves the problem in \eqref{homo_prob} and $Q_3=Q_3(0)$. Additionally, $A_0$ satisfies  \eqref{A-condition1} and \eqref{A-condition2}, replacing $A(z,\xi)$ with $A_0(\xi)$.

\begin{lemma}[Second comparison estimate]\label{lem:comp2}
Let $v$ be the weak solution to \eqref{freezing_prob}.
Then for any $\epsilon>0$, there exists a small $\delta=\delta(n,\nu,\Lambda,p,q,\epsilon)>0$ such that if \eqref{comp1_assumption} holds and $A$ is $(\delta,3)$-vanishing, then
\[%begin{equation}\label{comp2_result1}
\fint_{Q_3} \left| V(Dw)-V(Dv) \right|^2 dz \leq \epsilon
\]%end{equation}
and
\begin{equation}\label{comp2_result2}
\|Dv\|_{L^\infty(Q_1)} \leq N_0
\end{equation}
for some $N_0=N_0(n,\nu,\Lambda,p,q) \geq 1$, where $Q_1=Q_1(0)$.
\end{lemma}

\begin{proof}
We test \eqref{homo_prob} and \eqref{freezing_prob} with $\zeta=w-v$ using Steklov averages (cf. Remark~\ref{rmk:textfunction}) to observe that
\[
\int_{Q_{3}} \left( A(z,Dw) - A_0(Dv) \right) \cdot (Dw-Dv) \, dz \leq 0.
\]
It follows that
\[%begin{equation}\label{weakform_comp2}
\begin{split}
& \int_{Q_{3}} \left( A_0(Dw) - A_0(Dv) \right) \cdot (Dw-Dv) \, dz \\
& \qquad \leq -\int_{Q_{3}} \left( A(z,Dw) - A_0(Dw) \right) \cdot (Dw-Dv) \, dz.
\end{split}
\]%end{equation}
As $A_0(\cdot)$ is the mean of $A(z,\cdot)$ over $Q_{3}$ with respect to the $z$ variable,  \eqref{relation_A_V} holds, where $A(z,\cdot)$ is replaced with $A_0(\cdot)$.
Using this finding and Definition~\ref{def:vanishing}, we obtain the following:
\begin{equation}\label{comp2_pf_1}
\begin{split}
\fint_{Q_{3}} \left| V(Dw)-V(Dv) \right|^2 dz & \leq \fint_{Q_{3}} \left| A(z,Dw) - A_0(Dw) \right| |Dw-Dv| \, dz \\
& \leq \fint_{Q_{3}} \theta(A; Q_{3})(z) \, g(|Dw|) \, |Dw-Dv| \, dz \\
& \leq \fint_{Q_{3}} \theta(A; Q_{3})(z) \, g(|Dw|) \, |Dw| \, dz \\
& \qquad + \fint_{Q_{3}} \theta(A; Q_{3})(z) \, g(|Dw|) \, |Dv| \, dz \\
& =: \RN{1} + \RN{2}.
\end{split}
\end{equation}
To estimate $\RN{1}$, we use Proposition~\ref{prop0}, H\"{o}lder's inequality and the higher integrability (Theorem~\ref{thmhigh}) for $Dw$:
\[
\begin{split}
\RN{1} & = \fint_{Q_{3}} \theta(A; Q_{3})(z) \, g(|Dw|) \, |Dw| \, dz \\
& \le c \fint_{Q_{3}} \theta(A; Q_{3})(z) \, \phi(|Dw|) \, dz \\
& \le c \left( \fint_{Q_{3}} \theta^{\frac{1+\sigma}{\sigma}} \, dz \right)^{\frac{\sigma}{1+\sigma}} \left( \fint_{Q_{3}} \phi(|Dw|)^{1+\sigma} \, dz \right)^{\frac{1}{1+\sigma}} \\
& \le c \left( \fint_{Q_{3}} \theta^{\frac{1+\sigma}{\sigma}} \, dz \right)^{\frac{\sigma}{1+\sigma}} \left[(\phi \circ \mathcal{D}^{-1})\bigg(\fint_{Q_{4}}\phi(|Dw|)\,dz\bigg)\right]^{\frac{\sigma}{1+\sigma}} \left( \fint_{Q_{4}}\phi(|Dw|)\,dz \right)^{\frac{1}{1+\sigma}}.
\end{split}
\]
It follows from the $(\delta,2)$-vanishing condition of $A$ with \eqref{property_vanishing}, \eqref{comp1_pf_2} and \eqref{comp1_assumption} that
\begin{equation}\label{comp2_pf_2}
\RN{1} \le c \delta^{\frac{\sigma}{1+\sigma}} \left[(\phi \circ \mathcal{D}^{-1})(c)\right]^{\frac{\sigma}{1+\sigma}} c^{\frac{1}{1+\sigma}} \le c \delta^{\frac{\sigma}{1+\sigma}}.
\end{equation}

We also use Young's inequality, properties of $\phi$ and $\phi^*$, H\"{o}lder's inequality and the higher integrability for $Dw$ to estimate $\RN{2}$ as follows:
\[
\begin{split}
\RN{2} & = \fint_{Q_{3}} \theta(A; Q_{3})(z) \, g(|Dw|) \, |Dv| \, dz \\
& \leq c(\tau) \fint_{Q_{3}} \phi^* \left( \theta \, g(|Dw|) \right) dz + \tau \fint_{Q_{3}} \phi(|Dv|) \, dz \\
& \leq c(\tau) \fint_{Q_{3}} \max \left\{ \theta^{\frac{p}{p-1}}, \theta^{\frac{q}{q-1}} \right\} \, \phi^* \left( g(|Dw|) \right) dz + \tau \fint_{Q_{3}} \phi(|Dv|) \, dz \\
& \leq c(\tau) \fint_{Q_{3}} \left( \theta^{\frac{p}{p-1}} + \theta^{\frac{q}{q-1}} \right) \phi(|Dw|) \, dz + \tau \fint_{Q_{3}} \phi(|Dv|) \, dz \\
& \le c(\tau) \left( \fint_{Q_{3}} \theta^{\frac{p}{p-1} \cdot \frac{1+\sigma}{\sigma}} \, dz \right)^{\frac{\sigma}{1+\sigma}} \left( \fint_{Q_{3}} \phi(|Dw|)^{1+\sigma} \, dz \right)^{\frac{1}{1+\sigma}} \\
& \quad + c(\tau) \left( \fint_{Q_{3}} \theta^{\frac{q}{q-1} \cdot \frac{1+\sigma}{\sigma}} \, dz \right)^{\frac{\sigma}{1+\sigma}} \left( \fint_{Q_{3}} \phi(|Dw|)^{1+\sigma} \, dz \right)^{\frac{1}{1+\sigma}} + \tau \fint_{Q_{3}} \phi(|Dv|) \, dz
\end{split}
\]
for any $\tau \in (0,1)$.
By similar argument, we obtain 
\[
\RN{2} \le c(\tau) \, \delta^{\frac{\sigma}{1+\sigma}} + \tau \fint_{Q_{3}} \phi(|Dv|) \, dz.
\]
As in the proof of Lemma~\ref{lem:comp1}, we deduce that
\begin{equation}\label{comp2_Dv}
\fint_{Q_{3}} \phi(|Dv|) \, dz \leq c \fint_{Q_{3}} \phi(|Dw|) \, dz \leq c.
\end{equation}
Therefore, we observe that
\begin{equation}\label{comp2_pf_3}
\RN{2} \le c(\tau) \, \delta^{\frac{\sigma}{1+\sigma}} + \tau
\end{equation}
for any $\tau \in (0,1)$.

Combining \eqref{comp2_pf_1} with \eqref{comp2_pf_2} and \eqref{comp2_pf_3} results in the following:
\[
\fint_{Q_3} \left| V(Dw)-V(Dv) \right|^2 dz \le c(\tau) \, \delta^{\frac{\sigma}{1+\sigma}} + \tau.
\]
By setting $\tau \in (0,1)$ and $\delta \in (0,1)$ properly, we conclude that
\[
\fint_{Q_3} \left| V(Dw)-V(Dv) \right|^2 dz \leq \epsilon.
\]
Furthermore, the estimate in \eqref{comp2_result2} follows directly from Theorem~\ref{thmLip} and \eqref{comp2_Dv}.
\end{proof}

\section{Proof of Theorem~\ref{mainthm}}\label{sec4}

We prove the main theorem starting with the following technical lemma.

\begin{lemma}[\cite{Giusti_book}]\label{lem:technical}
Let $\Upsilon : [R_1,R_2] \rightarrow [0,\infty)$ be a bounded function. Suppose that for any $s_1$ and $s_2$, where $0 < R_1 \leq s_1 < s_2 \leq R_2$,
\[
\Upsilon(s_1) \leq \vartheta \Upsilon(s_2) + \frac{P}{(s_2-s_1)^{\kappa}} + Q
\]
with $P,Q \geq 0$, $\kappa>0$ and $\vartheta \in [0,1)$.
Then there exists $c=(\vartheta,\kappa)>0$ such that
\[
\Upsilon(R_1) \leq c(\vartheta,\kappa) \left[ \frac{P}{(R_2-R_1)^{\kappa}} + Q \right].
\]
\end{lemma}

\begin{proof}[\bf Proof of Theorem~\ref{mainthm}]   First, there exists $\tilde \psi\in \Phi_{\mathrm c}\cap C^1([0,\infty))$ satisfying \inc{p_1} and \adec{q_1} such that $\psi\approx \tilde \psi$ (see \cite[Remark 2.6]{HasOkparahigh}). 
Therefore, we assume that  $\psi\in \Phi_{\mathrm c}\cap C^1([0,\infty))$ without loss of generality. Moreover, we also assume that  $g(1)=\psi(1)=1$. 

We let $Q_{2R}=Q_{2R}(z_0)\Subset \Omega_I$ with $R\le R_0$ be fixed, where $R_0>0$ derives from the $(\delta,R_0)$-vanishing condition of $A$. For simplicity, we write $Q_\rho=Q_\rho(z_0)$, $\rho\in(0,2R]$. In addition, for $\rho>0$ and $\lambda>0$, we define
\[%begin{equation}\label{E}
E(\rho,\lambda):=\{z\in Q_{\rho}:|Du(z)|>\lambda\}.
\]%end{equation}
Also, we define the following:
\begin{equation}\label{lambda0}
\lambda_0 := \mathcal{D}^{-1} \left( \fint_{Q_{2R}} \left[ \phi(|Du|) + \frac{1}{\delta} \phi(|F|)\right] dz \right),
\end{equation}
where $\delta \in (0,1)$, depending only on $n$, $\nu$, $\Lambda$, $p$, $q$, $p_1$, $q_1$, and  $L$, is determined later (see below from \eqref{epsilonchoice}), and $\mathcal D$ is defined by  \eqref{D-def}. From \inc{p} of $\phi$ with $p>\frac{2n}{n+2}$,
\begin{equation} \label{Dineq2}
C^{\min\left\{2,\frac{p(n+2)-2n}{2}\right\}} \mathcal{D}(s)
=
\min\Big\{C^2, C^{\frac{p(n+2)-2n}{2}}\Big\} \mathcal{D}(s)
\le 
\mathcal{D}(Cs) 
%\leq  
%\max\left\{C^2, C^{\frac{q(n+2)-2n}{2}}L^{\frac{n+2}{2}}\right\} \mathcal{D}(t)
\end{equation}
for all $s>0$ and $C\ge 1$.

We let $R \leq r_1 < r_2 \leq 2R$ and consider any $\lambda$ satisfying the following:
\begin{equation}\label{lambda1}
\lambda \geq \lambda_1:= \left(\frac{64R}{r_2-r_1}\right)^
{(n+2)\max\left\{\frac{1}{2}, \frac2{p(n+2)-2n}\right\} }\lambda_0.
\end{equation}
With this $\lambda$, we also define
\begin{equation}\label{rlambda}
\rho_\lambda:= 
\min\{1,\phi_2(\lambda)^{\frac12}\}(r_2-r_1) \le R.
\end{equation}
We notice that  $Q^{\lambda}_{\rho}(z) \subset Q_{r_2}\subset Q_{2R}$ 
%$Q^{\lambda}_r(w)\subset Q_{2\rho}$ 
for any $z \in E(r_1,\lambda)$ and $\rho \leq \rho_\lambda$, where $Q^{\lambda}_{\rho}(z)$ is an intrinsic parabolic cylinder defined in Subsection~\ref{subsec2.1}. Then, we prove a Vitali type covering of the super-level set $E(r_1,\lambda)$ satisfying a balancing condition on each set.

\begin{lemma}\label{lem:covering}
For each $R \leq r_1<r_2\leq 2R$ and $\lambda\geq \lambda_1$, there exist $w_i\in E(r_1,\lambda)$ and $\rho_i\in \left(0,\frac{\rho_\lambda}{32} \right)$, $i=1,2,3,\cdots$, such that $Q^\lambda_{\rho_i}(w_i)$ are mutually disjoint,
\[%begin{equation}\label{stop0}
E(r_1,\lambda)\setminus N\ \subset\ \bigcup_{i=1}^\infty Q^\lambda_{8\rho_i}(w_i) 
\]%end{equation}
for some Lebesgue measure zero set $N$,
\[%begin{equation}\label{stop1}
\fint_{Q^\lambda_{\rho_i}(w_i)} \left[ \phi(|Du|) + \frac{1}{\delta} \phi(|F|) \right] dz=\phi(\lambda),
\]%end{equation}
and
\[%begin{equation}\label{stop2}
\fint_{Q^\lambda_{\rho}(w_i)} \left[ \phi(|Du|) + \frac{1}{\delta} \phi(|F|) \right] dz < \phi(\lambda) \ \ \ \text{for all } \rho \in(\rho_i,\rho_\lambda].
\]%end{equation}
\end{lemma}

\begin{proof}
The proof is exactly the same as that in \cite[Lemma 4.6]{HasOkparahigh}. However, for reader's convenience, we provide the proof. For $w\in E(r_1,\lambda)$ and $\rho \in \left[\frac{\rho_\lambda}{32},\rho_\lambda \right)$, using \eqref{lambda0}, we obtain
\[\begin{split}
\fint_{Q^\lambda_{\rho}(w)}\left[ \phi(|Du|) + \frac{1}{\delta} \phi(|F|) \right] dz & \leq \frac{|Q_{2R}|}{|Q^\lambda_{\rho}(w)|} \fint_{Q_{2R}}\left[ \phi(|Du|) + \frac{1}{\delta} \phi(|F|)  \right]dz  \\
& \leq \frac{|Q_{2R}|}{|Q_{\rho_\lambda/32}(w)|}\phi_2(\lambda)\mathcal{D}(\lambda_0).
\end{split}\]
It follows from \eqref{rlambda}, \eqref{Dineq2}, \eqref{lambda1}, and \eqref{D-def} that
%
%if $\phi_2(\lambda)\geq 1$, 
\[\begin{split}
\frac{|Q_{2R}|}{|Q_{\rho_\lambda/32}(w)|}\phi_2(\lambda)\mathcal{D}(\lambda_0) 
& \leq 
\bigg(\frac{64R}{r_2-r_1}\bigg)^{n+2} 
\max\left\{1,\phi_2(\lambda)^{-\frac{n+2}{2}}\right\} \phi_2(\lambda)\mathcal{D}(\lambda_0)\\
%& \leq 
%\tfrac{1}{2} L^{-\frac{n+2}{2}} \mathcal{D}(\lambda_1) \max\left\{1,[\phi_2(\lambda)]^{-\frac{n+2}{2}}\right\} \phi_2(\lambda)\\
&\leq \mathcal{D}(\lambda)  \max\left\{1,\phi_2(\lambda)^{-\frac{n+2}{2}}\right\} \phi_2(\lambda) \\
&= \min\left\{1,\phi_2(\lambda)^{\frac{n+2}{2}}\right\}\lambda^2\max\left\{1,\phi_2(\lambda)^{-\frac{n+2}{2}}\right\} \phi_2(\lambda)
= \phi(\lambda).
\end{split}\]
Therefore, we obtain
\[
\fint_{Q^\lambda_{\rho}(w)}\left[ \phi(|Du|) + \frac{1}{\delta} \phi(|F|) \right]dz\leq \phi(\lambda)\quad \text{for all }
\ \rho\in \left[ \tfrac{\rho_\lambda}{32},\rho_\lambda \right). 
\]
Moreover,  from the parabolic Lebesgue differentiation theorem, we deduce that, for almost every $w\in E(r_1,\lambda)$,
\[
\lim_{\rho \to 0^+}\fint_{Q^\lambda_{\rho}(w)}\left[ \phi(|Du|) + \frac{1}{\delta} \phi(|F|) \right] dz \geq \phi(|Du(w)|)> \phi(\lambda).
\]
As the map $\displaystyle \rho \mapsto \fint_{Q^\lambda_{\rho}(w)} \left[ \phi(|Du|) + \frac{1}{\delta} \phi(|F|) \right] dz$ is continuous, there exists $\rho_w\in \left(0,\frac{\rho_\lambda}{32}\right)$ such that
\[
\fint_{Q^\lambda_{\rho_w}(w)} \left[ \phi(|Du|) + \frac{1}{\delta} \phi(|F|) \right] dz=\phi(\lambda)
\]
and
\[
\fint_{Q^\lambda_{\rho}(w)} \left[ \phi(|Du|) + \frac{1}{\delta} \phi(|F|) \right] dz < \phi(\lambda) \quad \text{for all }\ \rho \in(\rho_w,\rho_\lambda].
\]
Applying Vitali's covering lemma for $\{Q^\lambda_{\rho_w}(w)\}$, we obtain the desired conclusion.
\end{proof}

 Continuing to prove Theorem~\ref{mainthm}, we set
 $Q_i^{(j)} := Q^\lambda_{2^j \rho_i}(w_i)$, $j=1,2,3,4,5$, we observe from the previous lemma that
\begin{equation}\label{main_pf_1}
|Q_i^{(0)}| \leq \frac{2}{\phi(\lambda)} \int_{Q_i^{(0)} \cap \{ |Du| > \frac{\lambda}{4} \}} \phi(|Du|) \, dz + \frac{2}{\delta \phi(\lambda)} \int_{Q_i^{(0)} \cap \{ |F| > \frac{\delta \lambda}{4} \}} \phi(|F|) \, dz
\end{equation}
and
\begin{equation}\label{main_pf_2}
\fint_{Q_i^{(5)}} \left[ \phi(|Du|) + \frac{1}{\delta} \phi(|F|) \right] dz < \phi(\lambda).
\end{equation}
Hereafter, we set $\lambda \geq \lambda_1$ and then
consider the following rescaled functions and vector fields:
\[
g_{\lambda}(s) := \frac{g(\lambda s)}{g(\lambda)} \quad \text{for} \ \, s \geq 0,
\]
\[
A_{\lambda, i}(z,\xi) := \frac{1}{g(\lambda)} A(W_i, \lambda \xi) \quad \text{for } \  z \in Q_{4}(0) \ \text{ and }\ \xi \in \R^n,
\]
\[
u_{\lambda, i}(z) := \frac{u(W_i)}{8\rho_i \lambda} \quad \text{and} \quad F_{\lambda, i}(z) := \frac{F(W_i)}{\lambda} \quad \text{for } \  z \in Q_{4}(0),
\]
where
\[
W_i=w_i+ \left( 8\rho_i x, \frac{\lambda}{g(\lambda)}(8\rho_i)^2 t \right) \quad \text{for} \ \, z=(x,t).
\]
Then, $g_{\lambda}$ satisfies the condition in \eqref{phi-condition} with the same constants $p$ and $q$, and $A_{\lambda, i}(z,\xi)$ satisfies \eqref{A-condition1} and \eqref{A-condition2} with $\Omega_I=Q_4(0)$ and is $(\delta,2)$-vanishing where  $g$ is replaced with $g_{\lambda}$. Moreover, 
we observe that $u_{\lambda, i}$ is a weak solution to
\[
\partial_t u_{\lambda, i}-\mathrm{div}\, A_{\lambda, i}(z,Du_{\lambda, i}) = -\div \left(\frac{g_{\lambda}(|F_{\lambda, i}|)}{|F_{\lambda, i}|}F_{\lambda, i}\right)\quad \text{in }\ Q_{4}(0).
\]
By defining
\[
\phi_{\lambda}(s):=\int_{0}^s g_{\lambda}(\sigma)\,d\sigma = \frac{1}{\lambda g(\lambda)}\phi(\lambda s), \quad s\ge 0,
\]
from \eqref{main_pf_2} where $\lambda g(\lambda) \geq \phi(\lambda)$ (see Proposition~\ref{prop0}), we obtain the following: 
\[
\fint_{Q_{4}(0)} \left[ \phi_{\lambda}(|Du_{\lambda, i}|) + \frac{1}{\delta} \phi_{\lambda}(|F_{\lambda, i}|) \right] dz = \frac{1}{\lambda g(\lambda)} \fint_{Q_i^{(5)}} \left[ \phi(|Du|) + \frac{1}{\delta} \phi(|F|) \right] dz < 1.
\]

Here, we let $\epsilon>0$ be sufficiently small, which is  determined in \eqref{epsilonchoice} below.  Then applying Lemmas~\ref{lem:comp1} and \ref{lem:comp2}, we can assert that there exist $\delta=\delta(\epsilon)>0$, $N_0 \geq 1$, and a function $\widetilde{v}_{\lambda, i}$ such that
\[
\fint_{Q_2(0)} \left| V_{\lambda}(Du_{\lambda, i})-V_{\lambda}(D\widetilde{v}_{\lambda, i}) \right|^2 dz \leq \epsilon \quad \text{and} \quad \|D\widetilde{v}_{\lambda, i}\|_{L^\infty(Q_1(0))} \leq N_0,
\]
where $V_{\lambda}(\xi)=\frac{1}{\sqrt{\lambda g(\lambda)}} \, V(\lambda \xi)$ for $\xi \in \R^n$.
Here we remark that both $\delta$ and $N_0$ are independent of $\lambda$ and $i$.
Finally we set
$v_{\lambda, i}(z)=v_{\lambda, i}(x,t):=8\rho_i \lambda \, \widetilde{v}_{\lambda, i} \left( \frac{x-y_i}{8\rho_i}, \frac{g(\lambda)(t-\tau_i)}{\lambda(8\rho_i)^2} \right)$, where $w_i=(y_i,\tau_i)$. Then we obtain
\begin{equation}\label{main_pf_3}
\fint_{Q_i^{(3)}} \left| V(Du)-V(Dv_{\lambda, i}) \right|^2 dz \leq \lambda g(\lambda) \epsilon, \quad \text{and} \quad \|Dv_{\lambda, i}\|_{L^\infty(Q_i^{(3)})} \leq N_0 \lambda.
\end{equation}

Next, for any $\lambda>\lambda_1$, we consider the upper-level sets $E(r_1, c_0 N_0 \lambda)$ with $c_0:= 2^{\frac{q}{p}}(c_*+1)^{\frac{1}{p}} \geq 1$, where $c_*>0$ is the hidden constant in \eqref{relation_V_2}.
By Lemma~\ref{lem:covering}, the collection $\{ Q_i^{(3)} \}$ covers $E(r_1,\lambda)\setminus N \supset E(r_1,c_0 N_0 \lambda)\setminus N$ with $|N|=0$,
which yields
\[%begin{equation}\label{main_pf_4}
\int_{E(r_1,c_0 N_0 \lambda)} \phi(|Du|) \, dz \leq \sum_{i=1}^{\infty} \int_{Q_i^{(3)} \cap \{ |Du|>c_0 N_0 \lambda\}} \phi(|Du|) \, dz.
\]%end{equation}
Using the convexity and \dec{q} of $\phi$, and \eqref{relation_V_2}, we obtain the following:
\[
\begin{split}
\phi(|Du|) & \leq \frac{1}{2} \phi(2|Du-Dv_{\lambda, i}|) + \frac{1}{2} \phi(2|Dv_{\lambda, i}|) \\
& \leq 2^{q-1} \phi(|Du-Dv_{\lambda, i}|) + 2^{q-1} \phi(|Dv_{\lambda, i}|) \\
& \leq 2^{q-1} c_* |V(Du)-V(Dv_{\lambda, i})|^2 + 2^{q-1} (c_*+1) \phi(|Dv_{\lambda, i}|).
\end{split}
\]
It follows from~\eqref{main_pf_3} that, on the set $Q_i^{(3)} \cap \{ |Du|>c_0 N_0 \lambda\}$,
\[
\begin{split}
\phi(|Du|) & \leq 2^{q-1} c_* |V(Du)-V(Dv_{\lambda, i})|^2 + 2^{q-1} (c_*+1) \phi(N_0 \lambda) \\
& \leq 2^{q-1} c_* |V(Du)-V(Dv_{\lambda, i})|^2 + \frac{1}{2} \phi(c_0 N_0 \lambda) \\
& \leq 2^{q-1} c_* |V(Du)-V(Dv_{\lambda, i})|^2 + \frac{1}{2} \phi(|Du|),
\end{split}
\]
and, hence,
\[
\phi(|Du|) \leq 2^q c_* |V(Du)-V(Dv_{\lambda, i})|^2.
\]
Combining this with~\eqref{main_pf_3} results in 
\begin{equation}\label{main_pf_5}
\int_{E(r_1,c_0 N_0 \lambda)} \phi(|Du|) \, dz \leq 2^q c_* \lambda g(\lambda) \epsilon \sum_{i=1}^{\infty} |Q_i^{(3)}|.
\end{equation}
Note that $|Q_i^{(3)}|=2^{3(n+2)}|Q_i^{(0)}|$ and that $\lambda g(\lambda) \lesssim \phi(\lambda)$ by Proposition~\ref{prop0}.
Therefore, the estimates \eqref{main_pf_5} and \eqref{main_pf_1} lead to the following:
\[
\int_{E(r_1,c_0 N_0 \lambda)} \phi(|Du|) \, dz \lesssim  \epsilon \sum_{i=1}^{\infty} \int_{Q_i^{(0)} \cap \{ |Du| > \frac{\lambda}{4} \}} \phi(|Du|) \, dz + \frac{\epsilon}{\delta} \sum_{i=1}^{\infty} \int_{Q_i^{(0)} \cap \{ |F| > \frac{\delta \lambda}{4} \}} \phi(|F|) \, dz,
\]
where the hidden constant depends only on $n$, $\nu$, $\Lambda$, $p$, and $q$.
As $Q_i^{(0)} \subset Q_{r_2}(z_0)$, $i=1,2,\dots$, are mutually disjoint, we conclude that, for any $\lambda \geq \lambda_1$,
\begin{equation}\label{main_pf_6}
\begin{split}
& \int_{E(r_1,c_0 N_0 \lambda)} \phi(|Du|) \, dz  \lesssim \epsilon
\int_{Q_{r_2} \cap \{ |Du| > \frac{\lambda}{4} \}} \phi(|Du|) \, dz + \frac{\epsilon}{\delta} \int_{Q_{r_2} \cap \{ |F| > \frac{\delta \lambda}{4} \}} \phi(|F|) \, dz.
\end{split}
\end{equation}

We set 
$$
|Du|_k:=\min\{|Du|,k\} \ \ \ \text{for }k\geq0.
$$
For $k > \lambda$, the inequality $|Du|_k>\lambda$ holds if and only if the inequality $|Du|>\lambda$ holds.
Therefore, for any $R\le r_1<r_2\le 2R$, 
\[
\begin{split}
 \int_{Q_{r_1}} \frac{\psi(\phi(|Du|_k))}{\phi(|Du|_k)} \phi(|Du|) \, dz & = \int_{Q_{r_1}} \int_{0}^{|Du|_k} d\hspace{-0.08cm} \left( \frac{\psi(\phi(\lambda))}{\phi(\lambda)} \right) \phi(|Du|) \, dz \\
& = \int_0^k \int_{Q_{r_1} \cap \{ |Du|_k > \lambda\}} \phi(|Du|) \, dz \, d\hspace{-0.08cm} \left( \frac{\psi(\phi(\lambda))}{\phi(\lambda)} \right) \\
& = \int_0^k \int_{E(r_1,\lambda)} \phi(|Du|) \, dz \, d\hspace{-0.08cm} \left( \frac{\psi(\phi(\lambda))}{\phi(\lambda)} \right) \\
& = \int_0^{\frac{k}{c_0 N_0}} \int_{E(r_1,c_0 N_0 \lambda)} \phi(|Du|) \, dz \, d\hspace{-0.08cm} \left( \frac{\psi(\phi(c_0 N_0 \lambda))}{\phi(c_0 N_0 \lambda)} \right) \\
& \leq \int_0^{\lambda_1} d\hspace{-0.08cm} \left( \frac{\psi(\phi(c_0 N_0 \lambda))}{\phi(c_0 N_0 \lambda)} \right) \int_{Q_{r_2}} \phi(|Du|) \, dz \\
&\qquad + \int_{\lambda_1}^{4k} \int_{E(r_1,c_0 N_0 \lambda)} \phi(|Du|) \, dz \, d\hspace{-0.08cm} \left( \frac{\psi(\phi(c_0 N_0 \lambda))}{\phi(c_0 N_0 \lambda)} \right) \\
& =: \RN{1} + \RN{2}.
\end{split}
\]
To estimate $\RN{1}$, we apply \adec{q_1} of $\psi$ and \dec{q} of $\phi$:
\[
\begin{split}
\RN{1} & \leq \frac{\psi(\phi(c_0 N_0 \lambda_1))}{\phi(c_0 N_0 \lambda_1)} \int_{Q_{r_2}} \phi(|Du|) \, dz \\
& \lesssim \left(\frac{64R}{r_2-r_1}\right)^
{(q_1-1)q(n+2)\max\left\{\frac{1}{2}, \frac2{p(n+2)-2n}\right\} } \psi_1(\phi(\lambda_0)) \int_{Q_{r_2}} \phi(|Du|) \, dz \\
& \lesssim \left(\frac{R}{r_2-r_1}\right)^
{\alpha_0} \psi_1(\phi(\lambda_0)) \int_{Q_{2R}} \phi(|Du|) \, dz,
\end{split}
\]
where $\alpha_0:=(q_1-1)q(n+2)\max\{\frac{1}{2},\frac{2}{p(n+2)-2n}\}$.

To estimate $\RN{2}$, we employ \eqref{main_pf_6} and Fubini's theorem:
\[
\begin{split}
\RN{2} & \lesssim \epsilon \int_{0}^{4k} \int_{Q_{r_2} \cap \{ |Du| > \frac{\lambda}{4} \}} \phi(|Du|) \, dz \, d\hspace{-0.08cm} \left( \frac{\psi(\phi(c_0 N_0 \lambda))}{\phi(c_0 N_0 \lambda)} \right) \\
& \quad + \frac{1}{\delta} \int_{0}^{\infty} \int_{Q_{r_2} \cap \{ |F| > \frac{\delta \lambda}{4} \}} \phi(|F|) \, dz\, d\hspace{-0.08cm} \left( \frac{\psi(\phi(c_0 N_0 \lambda))}{\phi(c_0 N_0 \lambda)} \right) \\
& = \epsilon \int_{0}^{k} \int_{Q_{r_2} \cap \{ |Du| > \lambda \}} \phi(|Du|) \, dz \, d\hspace{-0.08cm} \left( \frac{\psi(\phi(4c_0 N_0 \lambda))}{\phi(4c_0 N_0 \lambda)} \right) \\
& \quad + \frac{1}{\delta} \int_{0}^{\infty} \int_{Q_{r_2} \cap \{ |F| > \lambda \}} \phi(|F|) \, dz\, d\hspace{-0.08cm} \left( \frac{\psi(\phi(\frac{4}{\delta} c_0 N_0 \lambda))}{\phi(\frac{4}{\delta} c_0 N_0 \lambda)} \right) \\
& = \epsilon \int_{Q_{r_2}} \int_{0}^{|Du|_k} d\hspace{-0.08cm} \left( \frac{\psi(\phi(4c_0 N_0 \lambda))}{\phi(4c_0 N_0 \lambda)} \right) \phi(|Du|) \, dz  \\
& \quad + \frac{1}{\delta} \int_{Q_{r_2}} \int_{0}^{|F|} d\hspace{-0.08cm} \left( \frac{\psi(\phi(\frac{4}{\delta} c_0 N_0 \lambda))}{\phi(\frac{4}{\delta} c_0 N_0 \lambda)} \right) \phi(|F|) \, dz \\
& = \epsilon \int_{Q_{r_2}} \frac{\psi(\phi(4c_0 N_0 |Du|_k))}{\phi(4c_0 N_0 |Du|_k)} \phi(|Du|) \, dz + \frac{1}{\delta} \int_{Q_{r_2}} \frac{\psi(\phi(\frac{4}{\delta} c_0 N_0 |F|))}{\phi(\frac{4}{\delta} c_0 N_0 |F|)} \phi(|F|) \, dz \\
& \leq c \epsilon \int_{Q_{r_2}} \frac{\psi(\phi(|Du|_k))}{\phi(|Du|_k)} \phi(|Du|) \, dz + c \delta^{-1-(q_1-1)q} \int_{Q_{2R}} \psi(\phi(|F|)) \, dz.
\end{split}
\]
At this stage, we set $\epsilon=\epsilon(n,\nu,\Lambda,p,q,p_1,q_1,L)>0$ so small that 
\begin{equation}\label{epsilonchoice}
c\epsilon\leq \frac12.
\end{equation}
Thus, $\delta$ is chosen.
Combining the above estimates, we observe that
\[
\begin{split}
& \int_{Q_{r_1}} \frac{\psi(\phi(|Du|_k))}{\phi(|Du|_k)} \phi(|Du|) \, dz \\
& \ \leq \frac{1}{2} \int_{Q_{r_2}} \frac{\psi(\phi(|Du|_k))}{\phi(|Du|_k)} \phi(|Du|) \, dz + c \left(\frac{R}{r_2-r_1}\right)^
{\alpha_0} \psi_1(\phi(\lambda_0)) \int_{Q_{2R}} \phi(|Du|) \, dz \\
& \qquad + c \int_{Q_{2R}} \psi(\phi(|F|)) \, dz.
\end{split}
\]
Applying Lemma~\ref{lem:technical} yields the following:
\[
\begin{split}
& \int_{Q_R} \frac{\psi(\phi(|Du|_k))}{\phi(|Du|_k)} \phi(|Du|) \, dz  \leq  c  \psi_1(\phi(\lambda_0)) \int_{Q_{2R}} \phi(|Du|) \, dz + c \int_{Q_{2R}} \psi(\phi(|F|)) \, dz.
\end{split}
\]
Letting $k \to \infty$ and using Fatou's lemma, we conclude the following:
\[
\begin{split}
\int_{Q_R} \psi(\phi(|Du|)) \, dz & \lesssim  (\psi_1\circ\phi)(\lambda_0) \int_{Q_{2R}} \phi(|Du|) \, dz + \int_{Q_{2R}} \psi(\phi(|F|)) \, dz \\
& \lesssim (\psi_1 \circ \phi\circ \mathcal{D}^{-1} )\bigg( \fint_{Q_{2R}} \left[ \phi(|Du|) + \phi(|F|) \right] dz \bigg)\int_{Q_{2R}} \phi(|Du|) \, dz \\
& \qquad + \int_{Q_{2R}} \psi(\phi(|F|)) \, dz,
\end{split}
\]
which completes the proof.
\end{proof}

\section{Parabolic System}  \label{sec5}

We briefly discuss parabolic systems with a Uhlenbeck structure with a discontinuous coefficient. We consider the following parabolic system
\begin{equation}\label{parasystem}
(u_i)_t - \mathrm{div} \left(a(z) \frac{g(|Du|)}{|Du|}Du_i\right)=0 \quad \text{in }\ \Omega_I, \quad i=1,2,\cdots, N,
\end{equation} 
where $u=(u_1,\cdots, u_N)$, $N\ge 1$, $g:[0,\infty)\to [0,\infty)$ with $g \in C([0,\infty))\cap C^1((0,\infty))$ and  $g(0)=0$ satisfies 
\begin{equation}\label{phi-condition1}
p-1 \le \frac{s g'(s)}{g(s)} \le q-1, \ \ \ ^\forall s\ge0 \quad \text{for some }\ 2 \le p  \leq q.
\end{equation}
Further, $a:\Omega_I\to \R$ satisfies $\nu\le a(\cdot)\le \Lambda$ for some $0<\nu\le \Lambda$. For the systems we only consider the degenerate case (i.e., $p\ge 2$ in \eqref{phi-condition1}). In this case,  the $(\delta,R)$-vanishing condition is replaced with the following BMO-type condition on $a(\cdot)$:

\begin{definition}
Let $\delta,R>0$. $a:\Omega_I\to \R$ is said to be \textit{$(\delta,R)$-vanishing} if 
\[
\fint_{Q_{r,\rho}} |a(z)- (a)_{Q_{r,\rho}}| \,dz \le \delta
\]
for all $Q_{r,\rho}\subset \Omega_I$ with $r,\rho\in (0,R]$.
 \end{definition}

Then, we state a result that is the counterpart of Theorem \ref{mainthm} for the system in \eqref{parasystem}.

\begin{theorem}\label{thmsystem}
Let $g:[0,\infty)\to [0,\infty)$ with $g \in C([0,\infty))\cap C^1((0,\infty))$ and  $g(0)=0$ satisfy  \eqref{phi-condition1},  and $a:\Omega_I\to \R$ satisfy $\nu\le a(\cdot)\le \Lambda$ for some $0<\nu\le \Lambda$. 
In addition, let $F\in L^\phi_{\mathrm{loc}}(\Omega_I)$ with $\phi$ defined in \eqref{phidef}, and $\psi:[0,\infty)\to [0,\infty)$ be a weak $\Phi$-function satisfying \ainc{p_1} and \adec{q_1} for some $1<p_1\le q_1$ with the constant $L\ge 1$. 
There exists a small $\delta=\delta(n,\nu,\Lambda,p,q,p_1,q_1,L,g(1),\psi(1))>0$ such that if $a(\cdot)$ is $(\delta,R_0)$-vanishing for some $R_0>0$ and $u$ is a local weak solution to \eqref{parasystem}, then we have the following implication
$$
\phi(|F|) \in L^\psi_{\mathrm{loc}}(\Omega_I) \ \ \Longrightarrow\ \ \phi(|Du|) \in L^\psi_{\mathrm{loc}}(\Omega_I)
$$
with the following estimate: for any $Q_{2R}=Q_{2R}(z_0)\Subset \Omega_I $ with $R\le R_0$,
\[%begin{equation}\label{highesti}
\begin{split}
\fint_{Q_R} \psi(\phi(|Du|)) \,dz & \leq c \left[ \Psi \bigg( \fint_{Q_{2R}} \left[ \phi(|Du|) + \phi(|F|) \right] dz \bigg) \right] \fint_{Q_{2R}} \phi(|Du|) \, dz \\
& \qquad + c \fint_{Q_{2R}} \psi(\phi(|F|)) \, dz
\end{split}
\]%end{equation}
for some $c=c(n,\nu,\Lambda,p,q,p_1,q_1,L,g(1),\psi(1))>0$, where $Q_R=Q_R(z_0)$ and $\Psi$ is given in Theorem~\ref{mainthm}. 
\end{theorem}

\begin{proof}[Sketch of the proof]
The proof is exactly the same as in Theorem~\ref{mainthm} except for the essential modification concerned with changing from single equations to systems.
More precisely, in the proof, we replace $A(z,\xi)$ where $\xi \in \R^n$ with $a(z)|\xi|^{p-2}\xi$ where $\xi\in\R^{nN}$. 
Instead of Theorem~\ref{thmLip}, we take advantage of the Lipschitz regularity result for homogeneous parabolic systems in \cite[Lemma 3.2]{Cho18-2}, which is a modification of  \cite[Theorem 2.2]{DieScharSchwa19}. 
The higher integrability result \cite[Theorem~1.5]{HasOkparahigh} (Theorem~\ref{thmhigh} in this paper) is proved for general parabolic systems. 
All analyses in Sections~\ref{sec3} and \ref{sec4} are the same in both the equation and system cases and both the degenerate ($p\ge2$) and singular($p<2$) cases. Therefore, by repeating Sections~\ref{sec3} and \ref{sec4} and considering the replacement mentioned above, we prove the theorem.
\end{proof}

\begin{remark}
In the above theorem, we only consider the degenerate case (i.e., $p\ge 2$ in \eqref{phi-condition1}). 
The main reason is the lack of Lipschitz regularity when $p<2$. 
In \cite[Theorem 2.2]{DieScharSchwa19}, the Lipschitz regularity for the parabolic $\phi$-Laplace system \eqref{parasystem}, where $F\equiv 0$, $a(\cdot)\equiv 1$  and $\phi$ is defined in \eqref{phidef}, is proved under the assumption that the gradient of a weak solution is locally in $L^2$. This additional assumption is unclear from the definition of a weak solution when $p<2$.  We might also prove the theorem by replacing \eqref{phi-condition1} with \eqref{phi-condition} and assuming that $|Du|\in L^2_{\mathrm{loc}}(\Omega_I)$ (see \cite[Remark 1.1]{DieScharSchwa19}).
\end{remark}

%%%%%%%%%%%%%%%%%%%%%%%%%%%%%%%%%%%%%%%%%%%%%%%%%%%%%%%%%%%%%%%%%%%%%%%%
%%%%%%%%%%%%%%%%%%%%%%%%%%%%%%%%%%%%%%%%%%%%%%%%%%%%%%%%%%%%%%%%%%%%%%%%
%%%%%%%%%%%%%%%%%%%%%%%%%%%%%%%%%%%%%%%%%%%%%%%%%%%%%%%%%%%%%%%%%%%%%%%%

%\section*{Acknowledgment}

%%%%%%%%%%%%%%%%%%%%%%%%%%%%%%%%%%%%%%%%%%%%%%%%%%%%%%%%%%%%%%%%%%%%%%%%
%%%%%%%%%%%%%%%%%%%%%%%%%%%%%%%%%%%%%%%%%%%%%%%%%%%%%%%%%%%%%%%%%%%%%%%%
%%%%%%%%%%%%%%%%%%%%%%%%%%%%%%%%%%%%%%%%%%%%%%%%%%%%%%%%%%%%%%%%%%%%%%%%

\bibliographystyle{amsplain}

\end{document}